\documentclass[12pt,reqno]{amsart}
\usepackage{amsfonts,latexsym}
\usepackage{amssymb}
\usepackage{ifthen}
\usepackage{amscd}
\usepackage{amsxtra}
\usepackage{graphicx}
\usepackage{color}
\nonstopmode \numberwithin{equation}{section}
\makeatletter
\newcommand*\bigcdot{\mathpalette\bigcdot@{.5}}
\newcommand*\bigcdot@[2]{\mathbin{\vcenter{\hbox{\scalebox{#2}{$\m@th#1\bullet$}}}}}
\makeatother

\newtheorem{thm}{Theorem}
\newtheorem{lem}{Lemma}
\newtheorem{cor}{Corollary}[section]

\newtheorem{cl}{Claim}
\newtheorem{ca}{Case}
\newtheorem{sca}{Subcase}
\newtheorem{scl}{Subclaim}
\newtheorem{conj}{Conjecture}

\theoremstyle{definition}
\newtheorem{defn}{Definition}

\newtheorem{op}[equation]{Open Problem}
\newtheorem{ques}[equation]{Question}
\newtheorem{rem}{Remark}[section]
\newtheorem{exam}[equation]{Example}

\newcounter {own}
\def\theown {\thesection       .\arabic{own}}

\newenvironment{pf}[1][]{%
	\vskip 3mm
	\noindent
	\ifthenelse{\equal{#1}{}}%
	{{\slshape Proof. }}%
	{{\slshape #1.} }%
}%
{\qed\bigskip}

\newcounter{alphabet}

\newenvironment{Thm}[1][]{\refstepcounter{alphabet}%
	\bigskip%
	\noindent%
	{\bf Theorem \Alph{alphabet}}%
	\ifthenelse{\equal{#1}{}}{}{ (#1)}%
	{\bf .} \itshape}{\vskip 8pt}



\newenvironment{Lem}[1][]{\refstepcounter{alphabet}%
	\bigskip%
	\noindent%
	{\bf Lemma \Alph{alphabet}}%
	{\bf .} \itshape}{\vskip 8pt}

\newcommand{\C}{{\mathbb C}}
\newcommand{\D}{{\mathbb D}}

\newcommand{\T}{{\mathbb T}}

\newcommand{\R}{{\mathbb R}}

\def\be{\begin{equation}}
	\def\ee{\end{equation}}

\newcommand{\bee}{\begin{enumerate}}
	\newcommand{\eee}{\end{enumerate}}

\newcommand{\blem}{\begin{lem}}
	\newcommand{\elem}{\end{lem}}
\newcommand{\bthm}{\begin{thm}}
	\newcommand{\ethm}{\end{thm}}
\newcommand{\bcor}{\begin{cor}}
	\newcommand{\ecor}{\end{cor}}
\newcommand{\beg}{\begin{exam}}
	\newcommand{\eeg}{\end{exam}}
\newcommand{\begs}{\begin{examples}}
	\newcommand{\eegs}{\end{examples}}
\newcommand{\bdefe}{\begin{defn}}
	\newcommand{\edefe}{\end{defn}}
\newcommand{\bprob}{\begin{prob}}
	\newcommand{\eprob}{\end{prob}}
\newcommand{\bques}{\begin{ques}}
	\newcommand{\eques}{\end{ques}}
\newcommand{\bei}{\begin{itemize}}
	\newcommand{\eei}{\end{itemize}}
\newcommand{\bcon}{\begin{conj}}
	\newcommand{\econ}{\end{conj}}
\newcommand{\bop}{\begin{op}}
	\newcommand{\eop}{\end{op}}

\newcommand{\bca}{\begin{ca}}
	\newcommand{\eca}{\end{ca}}
\newcommand{\bsca}{\begin{sca}}
	\newcommand{\esca}{\end{sca}}

\newcommand{\bcl}{\begin{cl}}
	\newcommand{\ecl}{\end{cl}}

\newcommand{\bscl}{\begin{scl}}
	\newcommand{\escl}{\end{scl}}

\newcommand{\bcons}{\begin{conjs}}
	\newcommand{\econs}{\end{conjs}}
\newcommand{\bprop}{\begin{propo}}
	\newcommand{\eprop}{\end{propo}}
\newcommand{\br}{\begin{rem}}
	\newcommand{\er}{\end{rem}}
\newcommand{\brs}{\begin{rems}}
	\newcommand{\ers}{\end{rems}}
\newcommand{\bo}{\begin{obser}}
	\newcommand{\eo}{\end{obser}}
\newcommand{\bos}{\begin{obsers}}
	\newcommand{\eos}{\end{obsers}}
\newcommand{\bpf}{\begin{pf}}
	\newcommand{\epf}{\end{pf}}
\newcommand{\ba}{\begin{array}}
	\newcommand{\ea}{\end{array}}
\newcommand{\beq}{\begin{eqnarray}}
	\newcommand{\beqq}{\begin{eqnarray*}}
		\newcommand{\eeq}{\end{eqnarray}}
	\newcommand{\eeqq}{\end{eqnarray*}}

\newcommand{\ds}{\displaystyle}

\begin{document}
\bibliographystyle{amsplain}
\title [] {Schwarz Lemma for  mappings satisfying Biharmonic Equations}


\author{Adel Khalfallah}

\address{Department of Mathematics and Statistics, King Fahd University of Petroleum and
	Minerals, Dhahran 31261, Saudi Arabia}\email{adel.khalfallah@gmail.com}

\author{Fathi Haggui}
\address{Institut Préparatoire Aux Etude d'Ingénieurs de Monastir (IPEIM)\\Université de Monastir}
\email{fathi.haggui@gmail.com}

\author{Mohamed Mhamdi}
\address{Ecole supérieure des Sciences et  de la Technologie de Hammam Sousse (ESSTHS)\\Université de Sousse}
\email{mhamdimed7@gmail.com}

\begin{abstract}
In this paper, we establish some Schwarz type lemmas for mappings $\Phi$ satisfying the inhomogeneous
biharmonic Dirichlet problem  $ \Delta (\Delta(\Phi)) = g$ in $\D$,  $\Phi=f$ on $\T$ and $\partial_n \Phi=h$ on $\T$, where $g$ is a continuous function on $\overline{\D}$, $f,h$ are continuous functions on $\T$, where $\D$ is the unit disc  of the complex plane $\C$ and $\T=\partial \D$ is the unit circle. 
To reach our aim, we start by investigating  some properties of  $T_2$-harmonic functions. Finally, we prove   a Landau-type theorem.
\end{abstract}

\maketitle
\section{Preliminaries and Main Results}


Let $\mathbb{C}$ denote the complex plane and $\mathbb{D}$ the open unit disk in $\mathbb{C}$. Let $\mathbb{T}=\partial\mathbb{D}$ be the boundary of $\mathbb{D}$, and $\overline{\mathbb{D}} = \D \cup \mathbb{T}$, the closure of $\mathbb{D}$. Furthermore, we denote by $\mathcal{C}^m(\Omega)$ the set of all complex-valued $m-$times continuously differentiable functions from $\Omega$ into $\mathbb{C}$, where
$\Omega$ stands for a domain of $\mathbb{C}$ and $m \in \mathbb{N}$. In particular, $\mathcal{C}(\Omega):= \mathcal{C}^0(\Omega)$ denotes the set of all continuous functions in $\Omega$.\\
For a real $2 \times 2$ matrix $A$, we use the matrix norm
$$\| A\|=\sup\lbrace|Az| : |z|=1\rbrace,$$
and the matrix function
$$\lambda(A)=\inf\lbrace |Az| : |z|=1 \rbrace .$$
For $z = x+iy \in \C$, the formal derivative of a complex-valued function $\Phi = u+iv$
is given by\\
$$D_\Phi=\begin{pmatrix}
u_x && u_y\\
v_x && v_y
\end{pmatrix},$$
so that
$$\| D_\Phi\|=|\Phi_z|+|\Phi_{\overline{z}}|\ \ \ \text{and}\ \ \ \lambda(D_\Phi)=\big||\Phi_z|-|\Phi_{\overline{z}}|\big|,$$
where
$$\Phi_z=\frac{1}{2}(\Phi_x-i\Phi_y)\ \ \ \text{and}\ \ \ \ \Phi_{\overline{z}}=\frac{1}{2}(\Phi_x+i\Phi_y).$$
We use 
$$J_\Phi := \det D_\Phi = |\Phi_z|^2-|\Phi_{\overline{z}}|^2.$$

The main  objective of this paper is to establish a Schwarz-type lemma for  the solutions to the following inhomogeneous
biharmonic Dirichlet problem (briefly, IBDP):
\begin{equation}\label{ibdp}
\begin{cases}
\Delta^2 \Phi = g \ \ \  \text{in }\ \ \D,\\
\Phi\  = f \ \ \  \ \ \text{on}\ \ \mathbb{T},\\
\partial_n \Phi =h \ \ \ \ \text{on}\ \  \mathbb{T}.
\end{cases}
\end{equation}
where $$\ds\Delta=\frac{\partial^2}{\partial x^2}+\frac{\partial^2}{\partial y^2},$$ denotes the standard Laplacian and $\partial_n$ denotes the differentiation in the inward normal direction, $g \in \mathcal{C}(\overline{\D})$ and the boundary data $f$ and $h$ $\in\mathcal{C}(\mathbb{T})$.\\

We would like to mention that in \cite{Chel} and \cite{ChenZu}, the authors  have considered similar inhomogeneous biharmonic equations but with different boundaries conditions.\\
 
In order to state our main results, we introduce some necessary terminologies.
For $z,w \in \D$, let
$$G(z,w) = |z -w|^2\log\bigg|\frac{1 - z\overline{w}}{z - w}\bigg|^2- (1 - |z|^2)(1 - |w|^2),$$
and
$$P(z) = \frac{1-|z|^2}{|1 - z|^2}.$$
denote the biharmonic Green function and the harmonic Poisson kernel, respectively.\\
\\
By \cite[Theorem 1.1]{Li}, we see that all solutions 
of IBDP  (\ref{ibdp}) are given by

\begin{equation}
\Phi(z) = F_0[f](z) + H_0[h](z)-G[g](z),\nonumber
\end{equation}

\noindent where
$$F_0[f](z) =\frac{1}{2\pi}\int^{2\pi}_0 F_0(ze^{-i\theta})f(e^{i\theta}) d\theta,\quad H_0[h](z) =
\frac{1}{2\pi}\int^{2\pi}_0 H_0(ze^{-i\theta})h(e^{i\theta}) d\theta,$$
$$\text{and}\ \ \ G[g](z)=\frac{1}{16}\int_{\D}G(z,\omega)g(\omega)dA(\omega)$$
where $dA(\omega)$ denotes the Lebesgue area measure in $\D$.
Here the kernels $H_0$ and $F_0$ are given by
$$F_0(z) = H_0(z) + K_2(z),$$
$$ H_0(z) =\displaystyle\frac{1}{2}(1-|z|^2)P(z),$$ 
$$K_2(z)=\frac{1}{2}\frac{(1-|z|^2)^3}{|1 - z|^4}.$$

Thus, the solutions of the equation (\ref{ibdp}) are given by 

$$\Phi(z)=\frac{1}{2} (1-|z|^2)P[f+h](z)+K_2[f](z)-G[g](z). $$

Obviously  $P[f+h]$ is a bounded harmonic function, and  Heinz \cite{Hei} proved the following result, which is called the Schwarz lemma for planar 
harmonic functions: If $\Phi$ is a harmonic mapping from $\D$ into itself with $\Phi(0) = 0$, then for $z \in \D$,
\begin{equation*}
|\Phi(z)|\leq \frac{4}{\pi} \arctan |z|.
\end{equation*} 

Hethcote \cite{heth} and  Pavlovi\'c \cite[Theorem 3.6.1]{Pav} improved Heinz's result,  by removing the assumption $\Phi(0)=0$, and proved the following.

\begin{Thm}
	 Let $\Phi:\D \to  \D$ be a harmonic function from the unit disc to itself, then
\be \label{harm}
\bigg|\Phi(z)-\frac{1-|z|^2}{1+|z|^2} \Phi(0)\bigg|\leq \frac{4}{\pi} \arctan|z|,\ \ \ \ \ \ \ \ z\in \D.
\ee

\end{Thm}

\noindent Remark that $K_2[f]$ is a bounded $T_2$-harmonic which is a special type of biharmonic functions.  So naturally our first aim is to study the class of  $T_2$-harmonic functions.  Let us recall first  the definition of  $T_\alpha$-harmonic functions.
\begin{defn}\cite{Olf1}
Let $\alpha \in \R$, and let $f\in \mathcal{C}^2(\D)$. We say that $f$ is $T_\alpha$-harmonic if $f$ satisfies 
 $$T_\alpha(f)=0\quad \text{in} \ \D,$$
  \end{defn}
 where the $T_\alpha$-Laplacian operator is defined by
 $$T_\alpha=-\frac{\alpha^2}{4}(1-|z|^2)^{-(\alpha+1)}+\frac{1}{2}L_\alpha+\frac{1}{2}\overline{L_\alpha},$$
 with the weighted Laplacian operator $L_\alpha$ is defined by 
 $$L_\alpha=\frac{\partial}{\partial\overline{z}}(1 - |z|^2)^{-\alpha}\frac{\partial}{\partial z}.$$
 
 \begin{rem}
 	Let $f$ be a $T_\alpha$-harmonic function. 
 	\begin{enumerate}
 		\item If $\alpha=0$, then $f$ is harmonic.
 		\item If $\alpha=2(n-1)$, then  $f$ is $n$-harmonic,  where $n\in \{1,2,\ldots,\}$, see \cite{Abd,borh,Olf,Olf1}. 
 	\end{enumerate}
 \end{rem}
 
 The following result is the homogeneous  expansion of $T_\alpha$-harmonic functions.
 
 \begin{Thm}\cite{Olf}
 	Let $\alpha\in \R$ and $f\in \mathcal{C}^2(\D)$. Then $f$ is $T_\alpha$-harmonic if and only if it has a series expansion of the form 
 	\be\label{expansion}
 	f(z)=\sum_{k=0}^\infty c_k F(-\frac{\alpha}{2},k-\frac{\alpha}{2};k+1;|z|^2)z^k+
 	\sum_{k=1}^\infty c_{-k} F(-\frac{\alpha}{2},k-\frac{\alpha}{2};k+1;|z|^2)\overline{z}^k  \ee
 	for some sequence $\{c_k\}$ of complex numbers satisfying 
 	$$\limsup_{|k|\to \infty} |c_k|^{\frac{1}{|k|}}\leq 1.$$
 \end{Thm} 
$F$ is the {\it hypergeometric} function defined by the power series
$$F(a,b;c;x)=\sum_{n=0}^\infty \frac{(a)_n(b)_n}{(c)_n} \frac{x^n}{n!}, \quad |x|<1,$$
for $a,b,c\in \R$, with $c\not= 0,-1,-2,\ldots.,$ where $(a)_0=1$ and $(a)_n=a(a+1)\ldots (a+n-1)$ for $n=1,2,\ldots$ are the {\it Pochhammer} symbols.\\

We refer the reader to \cite{Olf}, where Olofsson gave a Poisson type  integral representation theorem for $T_\alpha$-harmonic mappings, for $\alpha>-1$.\\

It is well known that the Schwarz lemma is one of the most influential results  in many branches of mathematical research for more than a hundred years. We refer the reader to \cite{Bur,Chel,Kal1,Kra,Liu,Mate} for generalizations and applications of this lemma.\\

In this context, we establish  a Schwarz type  lemma for  $T_2$-harmonic functions.

\begin{thm}\label{theo1}
Let $u:\D\longrightarrow \D$ be a  $T_2$-harmonic function, then
\be\label{t2}
\bigg|u(z)-\frac{(1-|z|^2)^3}{(1+|z|^2)^2}u(0)\bigg|\leq \frac{2}{\pi}\Bigg[(1+|z|^2)\arctan |z|+ \frac{|z|(1-|z|^2)}{1+|z|^2}\Bigg],\nonumber
\ee
for all $z\in \D$.
\end{thm}


Next, we prove a Schwarz-Pick lemma for $T_2$-harmonic functions.
\begin{thm}\label{theo2}
	Let $u:\D\longrightarrow \D$ be a  $T_2$-harmonic function, then 
		\begin{equation}\label{theoe2}
	\| D_u(z)\| \leq \frac{(2+5|z|)(1+|z|^2)}{1-|z|^2}, \text{ for all } z\in \D. 
	\end{equation}
	\begin{equation}\label{grad0}
	\| D_u(0)\|\leq \frac{4}{\pi}.
	\end{equation} 
	Moreover, the inequality (\ref{grad0}) is sharp.
\end{thm}


Next we establish a Landau type theorem for $T_2$-harmonic functions
\begin{thm}\label{landau}
	Let $u\in \mathcal{C}^2(\D)$ be a $T_2$-harmonic function satisfying $u(0)=|J_u(0)|-1=0$ and $\sup_{z\in \D}|f(z)|\leq M$, where $M>0$ and $J_u$ is the Jacobian of $u$. Then $u$ is univalent on $D_{r_0}$, where $r_0$ satisfies the following equation
	$$\frac{\pi}{4M}-\frac{4}{\pi}\frac{M r_0}{(1-r_0)^3}\big(-r_0^3+3r_0^2-3r_0+3 \big)=0.$$
	Moreover, $u(\D_{r_0})$ contains a univalent disk $D_{R_0}$ with
		$$R_0 \geq  \frac{4M r_0^2}{\pi(1-r_0)^3}\big[-\frac{4}{5}r_0^3+\frac{9}{4}r_0^2-2r_0+\frac{3}{2}\big]. $$
\end{thm}

Now we are in the position to prove our main results
\begin{thm}\label{theo3}
Let $g\in \mathcal{C}(\overline{\D})$, $f, h \in \mathcal{C}(\mathbb{T})$ and suppose that $\Phi\in \mathcal{C}^4(\D) \cap \mathcal{C}(\overline{\D})$
satisfies (\ref{ibdp}). Then for $z \in \D$,
\begin{eqnarray}\label{sch}
\bigg| \Phi(z)-\frac{1}{2}\frac{(1-|z|^2)^3}{(1+|z|^2)^2}P[f](0)-\frac{1}{2}\frac{(1-|z|^2)^2}{1+|z|^2}P[f+h](0)\bigg|&\leq & \Big[\frac{2}{\pi}(1-|z|^2)\arctan|z|\Big]\| f+h\|_\infty\nonumber\\
&+&\frac{2}{\pi}\bigg[(1+|z|^2)\arctan |z|+|z|\frac{1-|z|^2}{1+|z|^2}\bigg]\| f\|_\infty\nonumber\\
&+&\frac{(1-|z|^2)^2}{64}\,\| g\|_\infty.
\end{eqnarray}
where $\| f \|_\infty= \ds\sup_{\zeta\in \mathbb{T}}|f(\zeta)|$, $\ds\| f+h \|_\infty= \sup_{\zeta\in \mathbb{T}}|f(\zeta)+h(\zeta)|$ and  $\ds\| g \|_\infty= \sup_{\zeta\in \mathbb{D}}|g(\zeta)|$.
\end{thm}

\begin{normalsize}
\textbf{Remark 1.2}
\end{normalsize}
Under the hypothesis of Theorem \ref{theo3}, if $g\equiv 0$, then $\Phi$ is biharmonic mapping and the estimate (\ref{sch}) can be written in the following form
\begin{eqnarray*}
\bigg| \Phi(z)-\frac{1}{2}\frac{(1-|z|^2)^3}{(1+|z|^2)^2}P[f](0)-\frac{1}{2}\frac{(1-|z|^2)^2}{1+|z|^2}P[f+h](0)\bigg|&\leq & \Big[\frac{2}{\pi}(1-|z|^2)\arctan|z|\Big]\| f+h\|_\infty \\
&+&\frac{2}{\pi}\bigg[(1+\!|z|^2)\arctan |z|+|z|\frac{1-|z|^2}{1+|z|^2}\bigg]\| f\|_\infty.
\end{eqnarray*}

\begin{thm}\label{theo4}
Let $g\in \mathcal{C}(\overline{\D})$, $f$ and $h \in \mathcal{C} (\mathbb{T})$. Suppose that $\Phi\in \mathcal{C}^4(\D)$ is  satisfying (\ref{ibdp}). Then for all $z\in \D$
\be
\| D_\Phi (z)\| \leq  \frac{2+5|z|}{1-|z|^2}(1+|z|^2)\| f\|_\infty
+(\frac{2}{\pi} +|z| )\| f+h\|_\infty +\frac{23}{48}\| g\|_\infty.
\ee
Moreover at $z=0$, we have
\be
\| D_\Phi (0)\| \leq  \frac{4}{\pi}\| f\|_\infty
+\frac{2}{\pi}\| f+h\|_\infty +\frac{23}{48}\| g\|_\infty.
\ee
\end{thm}
The classical Schwarz lemma at the boundary is as follows.

\begin{Thm}
Suppose $f:\D\longrightarrow\D$ is a holomorphic function with $f(0)=0$, and further, $f$ is analytic at $z=1$ with $f(1)=1$. Then, the following two conditions hold:
\begin{enumerate}
\item[(a)] $f'(1)\geq 1;$
\item[(b)]  $f'(1)=1$ if and only if $f(z)=z.$
\end{enumerate}
\end{Thm}

Which is known as the Schwarz lemma on the boundary, and its generalizations have important applications in geometric theory of functions (see, \cite{Gol,Lav,Pom}). Among the recent papers devoted to this subject, for example, Burns and Krantz \cite{Bur}, Krantz \cite{Kra}, Liu and Tang \cite{Liu} explored many versions of the Schwarz lemma at the boundary point of holomorphic functions, Dubinin also  applied this latter for algebraic polynomials and rational functions (see \cite{Dub,Dub1}). In the present paper, we refine the Schwarz type
lemma at the boundary for $\Phi$ satisfies (\ref{ibdp}) as an application of Theorem \ref{theo3}. 
\begin{thm}\label{theo5}
Suppose that $\Phi\in \mathcal{C}^4(\D)\cap\mathcal{C}(\overline{\D})$ satisfies (\ref{ibdp}), where $g\in \mathcal{C}(\overline{\D})$ and $f$, $h\in \mathcal{C}(\mathbb{T})$ such that $\|f\|_\infty \leq 1$, and  $\| f+h\|_\infty\leq 1$. 
If $\ds\lim_{r\to 1} |\Phi(r\eta)|=1$  for  $\eta\in \mathbb{T}$, then 
\begin{equation}
\liminf_{r\to 1}\frac{|\Phi(\eta)-\Phi(r\eta)|}{1-r}\geq 1-\| f+h\|_\infty.\nonumber
\end{equation}
In particular if $\|f+h\|_\infty=0$,  then
\begin{equation*}
\liminf_{r\to 1}\frac{|\Phi(\eta)-\Phi(r\eta)|}{1-r}\geq 1,
\end{equation*}
and this estimate is sharp.
\end{thm}

Denote by $\mathcal{H}(\D)$ the set of all holomorphic functions $\Phi$ in $\D$ satisfying the standard normalization: $\Phi(0) = \Phi'(0)-1 = 0$. In the early 20th century, Landau \cite{Lan} showed
that there is a constant $r > 0$, independent of elements in $\mathcal{H}(\D)$, such that $\Phi(\D)$ contains a disk of radius $r$. Later, the Landau theorem has become an important
tool in geometric function theory.  To establish analogs of the Landau type theorem for more general classes of functions, it is necessary to
restrict our focus on certain subclasses (cf. \cite{Abd}, \cite{Boc}, \cite{Bonk}, \cite{Che1}, \cite{Che2},
\cite{Chep}, \cite{Chev}). \\
\\
For convenience, we make a notational convention: for $g\in\mathcal{C}(\overline{\D})$ and $h \in \mathcal{C}(\mathbb{T})$, let
$\mathcal{BF}_{g,h}(\overline{\D})$ denote the class of all complex-valued functions $\Phi\in \mathcal{C}^4(\D)\cap \mathcal{C}(\overline{\D})$ satisfying
(\ref{ibdp}) with the normalization $\Phi(0) = J_\Phi(0) - 1 = 0$.\\

We establish the following Landau-type theorem for $\Phi \in \mathcal{BF}_{g,h}(\overline{\D})$. In particular, if $g \equiv 0$, then $\Phi \in \mathcal{BF}_{g,h}(\overline{\D})$  is biharmonic. In this
sense, the following result is a generalization of \cite[Theorem 1]{Abd} and \cite[Theorem 2]{Abd}.
\begin{thm}\label{theo6}
Suppose that $M_1>0$, $M_2>0$ and $M_3>0$ are constants, and suppose that 
$\Phi \in \mathcal{BF}_{g,h}(\overline{\D})$ satisfies the following conditions:
$$\sup_{z\in \mathbb{T}} |f(z)|\leq M_1,\ \ \ \sup_{z\in \mathbb{T}} |f(z)+h(z)|\leq M_2,\ \ \text{ and }\sup_{z\in \D} |g(z)|\leq M_3.$$
Then $\Phi$ is univalent in $\D_{r_0}$, and $\Phi(\D_{r_0})$ contains a univalent disk $\D_{R_0}$,where $r_0$ satisfies
the following equation:
$$(\frac{4}{\pi} M_1+\frac{2}{\pi}M_2+\frac{23}{48}M_3)\sigma(r_0)=1,$$
with
\begin{eqnarray*}
\sigma(|z|)&:=&(M_1+M_2+\frac{101}{120} M_3)|z|+\frac{2M_2|z|}{\pi}\bigg[\frac{(2-|z|)(1+|z|^2)}{(1-|z|)^2}+|z| \bigg]\\
&+&\frac{4M_1|z|}{\pi(1-|z|)^3}\big(-|z|^3+3|z|^2-3|z|+3\big).
\end{eqnarray*}

and
$$R_0\geq \frac{r_0}{\frac{8}{\pi} M_1+\frac{4}{\pi}M_2+\frac{23}{24}M_3}. $$

\end{thm}
\section{Schwarz and Landau Type Lemmas for $T_2$-harmonic functions}
\subsection{Schwarz Type Lemma for $T_2$-harmonic functions}\hfill

The main purpose of this section is to prove Schwarz type lemma   for $T_2$-harmonic functions. 

\begin{proof}[Proof of Theorem \ref{theo1} ]
Let $0\leq r=|z|<1$. As $u$ is a $T_2$-harmonic function, then 
$$ u(z)=\frac{1}{2\pi}\int_0^{2\pi}\frac{1}{2}\frac{(1-r^2)^3}{|1-ze^{-i\theta}|^4}u^*(e^{i\theta})d\theta,$$
where $u^*\in L^\infty(\mathbb{T})$. Thus
\begin{eqnarray*}
\bigg|u(z)-\frac{(1-r^2)^3}{(1+r^2)^2}u(0)\bigg|&\leq &\frac{1}{4\pi}\int_0^{2\pi}\bigg|\frac{(1-r^2)^3}{(1+r^2-2r\cos\theta)^2}-\frac{(1-r^2)^3}{(1+r^2)^2}\bigg|d\theta \\
&=&\frac{1}{2\pi}\int_0^{\pi}\bigg|\frac{(1-r^2)^3}{(1+r^2-2r\cos\theta)^2}-\frac{(1-r^2)^3}{(1+r^2)^2}\bigg|d\theta\\
&=&\frac{1}{2\pi}\bigg[\int_0^{\pi/2}\frac{(1-r^2)^3}{(1+r^2-2r\cos\theta)^2}-\frac{(1-r^2)^3}{(1+r^2)^2}d\theta\\
&-&\int_{\pi/2}^{\pi}\frac{(1-r^2)^3}{(1+r^2-2r\cos\theta)^2}-\frac{(1-r^2)^3}{(1+r^2)^2}d\theta\bigg]\\
&=&\frac{1}{2\pi}\bigg[\int_0^{\pi/2}\frac{(1-r^2)^3}{(1+r^2-2r\cos\theta)^2}d\theta -\int_{\pi/2}^{\pi}\frac{(1-r^2)^3}{(1+r^2-2r\cos\theta)^2}d\theta\bigg]\\
&=&\frac{1}{2\pi}\bigg[2J(\pi/2)-J(\pi)\bigg],
\end{eqnarray*}

where 

$$ J(\theta):=\int_0^{\theta}\frac{(1-r^2)^3}{(1+r^2-2r\cos\varphi)^2}d\varphi.$$
Easy but tedious computations show that 
\begin{lem}\label{lem1}
For $0\leq\theta<\pi$, and $r\in[0,1)$, we have 
$$J(\theta)=\frac{2r(1-r^2)\sin\theta}{1+r^2-2r\cos\theta}+2(1+r^2)\arctan\bigg(\frac{(1+r)\tan\theta/2}{1-r}\bigg),$$
and  $\ds J(\pi)= \lim_{\theta \to \pi} J(\theta)= \pi(1+r^2).$
\end{lem}

Then by Lemma \ref{lem1}, and using the fact that $ \ds \arctan(\frac{1+r}{1-r})-\frac{\pi}{4}=\arctan r$,  we have
\begin{eqnarray*}
\bigg|u(z)-\frac{(1-r^2)^3}{(1+r^2)^2}u(0)\bigg|&\leq &\frac{1}{2\pi}\bigg[2J(\pi/2)-J(\pi)\bigg]\\
&=&\frac{1}{2\pi}\bigg[4\frac{r(1-r^2)}{1+r^2}+4(1+r^2)\arctan(\frac{1+r}{1-r})-\pi(1+r^2)\bigg]\\
&=&\frac{2}{\pi}\bigg[\frac{r(1-r^2)}{1+r^2}+(1+r^2)\arctan r\bigg].
\end{eqnarray*}
\end{proof}

To prove Theorem \ref{theo2}, we need the following lemma


\begin{Lem}{\cite{Li}}\label{theod}
For any $z\in \D$, we have
$$I_\alpha(z):=\frac{1}{2\pi}\int^{2\pi}_0\frac{d\theta}{|1-ze^{i\theta}|^{2\alpha}}=\sum^{+\infty}_{n=0}\bigg(\frac{\Gamma(n+\alpha)}{n!\Gamma(\alpha)}\bigg)^2|z|^{2n},$$ 
where $\alpha >0$ and $\Gamma$ denotes the Gamma function.
\end{Lem}
Thus
\begin{equation}\label{j2}
I_2(z)=\sum_{n=0}^\infty (n+1)^2|z|^{2n}=\frac{1+|z|^2}{(1-|z|^2)^3}.
\end{equation}

\begin{proof}[Proof of Theorem \ref{theo2}]
	
As $u$  is a bounded $T_2$-harmonic function, then
$$ u(z)= \frac{1}{2\pi}\int_0^{2\pi} K_2(ze^{-i\theta})u^*(e^{i\theta})d\theta,$$
where $u^*\in L^\infty(\mathbb{T})$. Elementary computations show that 

\begin{eqnarray}\label{uz0}
u_z(z)&=&\frac{1}{2\pi}\int^{2\pi}_0 [K_2(ze^{-i\theta})]_z u^*(e^{i\theta})d\theta\nonumber \\
&=&
 \frac{1}{2\pi}\int^{2\pi}_0\frac{(1-|z|^2)^2[2e^{-i\theta}(1-|z|^2)-3\overline{z}(1-ze^{-i\theta})]}{2(1-\overline{z}e^{i\theta})^2(1-ze^{-i\theta})^3}u^*(e^{i\theta})d\theta.
\end{eqnarray}
Hence 
\begin{eqnarray*}
|u_z(z)|
&\leq & \frac{(1-|z|^2)^2}{2}\bigg[\frac{1}{2\pi}\int^{2\pi}_0\bigg|\frac{2e^{-i\theta}(1-|z|^2)}{(1-\overline{z}e^{i\theta})^2(1-ze^{-i\theta})^3}\bigg|d\theta +\frac{1}{2\pi}\int^{2\pi}_0\bigg|\frac{3\overline{z}(1-ze^{-i\theta})}{(1-\overline{z}e^{i\theta})^2(1-ze^{-i\theta})^3}\bigg|d\theta\bigg]\\
&\leq & \frac{(1-|z|^2)^2}{2}\bigg[2(1-|z|^2)I_{5/2}(z)+3|z|I_2(z)\bigg]\\
&\leq & \frac{(1-|z|^2)^2}{2}(2+5|z|)I_2(z).
\end{eqnarray*}
The last inequality follows from the following estimate  $\displaystyle I_{5/2}(z) \leq \frac{I_2(z)}{1-|z|}$.\\
Using the explicit expression of $I_2$,  see (\ref{j2}),  we obtain

\be
 |u_z(z)|\leq  \frac{(2+5|z|)(1+|z|^2)}{2(1-|z|^2)}.\nonumber
  \ee
Similarly,
\be
|u_{\overline{z}}(z)| \leq  \frac{(2+5|z|)(1+|z|^2)}{2(1-|z|^2)}.\nonumber
\ee
Thus
$$
	\| D_u(z)\| =|u_z(z)|+|u_{\overline{z}}(z)|\leq  \frac{(2+5|z|)(1+|z|^2)}{1-|z|^2}.$$

Next let us show the estimate (\ref{grad0}). From Theorem \ref{theo1}, we deduce that near $0$
\be\label{gradient0}
 |u(z)-u(0)| \leq \frac{4}{\pi}|z|+O(|z|^2). 
\ee
Indeed, near 0,  $$  (1+r^2)\arctan r+ r\frac{1-r^2}{1+r^2}=2r+O(r^2), \text{ and }$$
$$\frac{(1-r^2)^3}{(1+r^2)^2}= 1+O(r^2).$$
Hence from (\ref{gradient0}), we get
$$\|D_u(0)\| \leq \frac{4}{\pi}.$$

To show that the last estimate is sharp. Let us consider the $T_2$-harmonic mapping defined by
$$U(z)=K_2[\chi_{\T^+}- \chi_{\T^-}](z),$$
where $\chi_{\T^+}$ (resp. $\chi_{\T^-}$) denotes the characteristic function of the upper (resp. lower) unit circle $\T$.

By  (\ref{uz0}), we have
$$U_z(0)= \frac{1}{2\pi}  \int_0^{2\pi} e^{-i\theta}
(\chi_{\T^+}- \chi_{\T^-})(\theta) d\theta=\frac{1}{2\pi}  \int_0^{\pi} e^{-i\theta}d\theta-\frac{1}{2\pi}  \int_\pi^{2\pi} e^{-i\theta}d\theta=-\frac{2i}{\pi}.$$
Hence
$$|\nabla U(0)|=2|U_z(0)|=\frac{4}{\pi}.$$
\end{proof}

\subsection{Landau type theorem for $T_2$- harmonic functions}\hfill

We will use the following theorem provides some estimates on the coefficients of $T_\alpha$-harmonic mappings.

\begin{Thm}\cite{Chev}
	For $\alpha>-1$, let $u\in \mathcal{C}^2(\D)$ be a $T_\alpha$-harmonic function with the series expansion of the form (\ref{expansion}) and $\sup_{z\in \D}|u(z)| \leq M$, where $M>0$. Then, for $k\in \{1,2,\ldots\}$,
	\be\label{2.6}
	\left|c_kF(-\frac{\alpha}{2},k-\frac{\alpha}{2};k+1;1)\right|+\left|c_{-k}F(-\frac{\alpha}{2},k-\frac{\alpha}{2};k+1;1)\right| \leq \frac{4M}{\pi},
	\ee
	and
	\be\label{c0}
	|c_0F(-\frac{\alpha}{2},-\frac{\alpha}{2};1;1)|\leq M.
	\ee 
\end{Thm}

\begin{proof}[Proof of Theorem \ref{landau}]

	As $u$ is $T_2$-harmonic function on the unit disk with $u(0)=0$, then by plugging $\alpha=2$ in  (\ref{expansion}), we have $c_0=0$ and 
	$$ u(z)=\sum_{k=1}^\infty c_k F(-1,k-1;k+1;|z|^2)z^k+
	\sum_{k=1}^\infty c_{-k} F(-1,k-1;k+1;|z|^2)\overline{z}^k.$$

An elementary computation shows that the function $\displaystyle F(-1,k-1;k+1;\bigcdot)$ is given by 
\be\label{26}
F(-1,k-1;k+1;w)=1-\frac{k-1}{k+1}w, \quad w\in [0,1).\\
\ee

Clearly, the function $w \mapsto F(-1,k-1;k+1;w)$ is decreasing on $[0,1]$ for $k \geq 1$. Thus

\be \label{2.7}
\frac{2}{k+1} \leq F(-1,k-1;k+1;w) \leq 1 \quad  \text { for } k\geq 1 \text{ and } w\in [0,1].\\ 
\ee 

Combining  (\ref{2.6})  for $\alpha=2$ and (\ref{2.7}), we obtain

\be\label{ck}
|c_k|+|c_{-k}|\leq \frac{2M}{\pi}(k+1) \quad \text{ for } k\geq 1.
\ee

Using (\ref{26}), we see that the mapping $u$ is given by 
\be
u(z)=\sum_{k=1}^\infty c_kz^k+c_{-k}\overline{z}^k -\sum_{k=1}^\infty \frac{k-1}{k+1}\bigg(c_kz^{k+1}\overline{z}+c_{-k}\overline{z}^{k+1}z\bigg).\nonumber
\ee

Therefore

\be\label{uz}
u_z(z)-u_z(0)= \sum_{k=2}^\infty kc_kz^{k-1}-\sum_{k=2}^\infty \frac{k-1}{k+1}\bigg( (k+1)c_kz^k\overline{z}+c_{-k}\overline{z}^{k+1} \bigg).
\ee

Similarly

\be\label{uzbar}
u_{\overline{z}}(z)-u_{\overline{z}}(0)= \sum_{k=2}^\infty kc_{-k}\overline{z}^{k-1}-\sum_{k=2}^\infty \frac{k-1}{k+1}\bigg( (k+1)c_{-k}\overline{z}^k z+c_{k}z^{k+1} \bigg).
\ee

Applying (\ref{uz}), (\ref{uzbar}) and (\ref{ck}), we obtain
\begin{eqnarray}
& & |u_z(z)-u_z(0) |+| u_{\overline{z}}(z)-u_{\overline{z}}(0)|\nonumber \\
&\leq& \sum_{k=2}^\infty  k \left(|c_k|+|c_{-k}|\right)|z|^{k-1}  \nonumber\\
&+&  \sum_{k=2}^\infty (k-1)\left(|c_k|+|c_{-k}| \right)|z|^{k+1}\nonumber\\
&+&  \sum_{k=2}^\infty \frac{k-1}{k+1}\left(|c_k|+|c_{-k}| \right)|z|^{k+1}\nonumber\\
&\leq& \frac{2M}{\pi}\bigg(\sum_{k=2}^\infty k(k+1)|z|^{k-1} +\sum_{k=2}^\infty (k-1)(k+1)|z|^{k+1}+\sum_{k=2}^\infty (k-1)|z|^{k+1}  \bigg)\nonumber\\
&=& \frac{2M|z|}{\pi(1-|z|)^2}\bigg( 2\frac {|z|^2-3|z|+3}{1-|z|}+\frac{|z|^2(3-|z|)}{1-|z|}+|z|^2\bigg).\nonumber
\end{eqnarray}

That is
\be\label{uzzz}
 |u_z(z)-u_z(0) |+| u_{\overline{z}}(z)-u_{\overline{z}}(0)|
\leq \frac{4M|z|}{\pi(1-|z|)^3}\big(-|z|^3+3|z|^2-3|z|+3\big).
\ee
Moreover, the RHS of (\ref{uzzz}) is strictly increasing as a function of $|z|$.\\

Applying Theorem \ref{theo2} (\ref{grad0}), we get
\be
1=J_u(0)=\|D_u(0)\| \lambda(D_u(0)) \leq \frac{4M}{\pi}\lambda(D_u(0)),\nonumber
\ee
which gives that
\be\label{ldu}
 \lambda\left(D_u(0)\right) \geq \frac{\pi}{4M}. 
\ee
We will show that $u$ is univalent in $\D_{r_0}$, where $r_0$ satisfies the following equation
	$$\frac{\pi}{4M}-\frac{4}{\pi}\frac{M r_0}{(1-r_0)^3}\big(-r_0^3+3r_0^2-3r_0+3 \big)=0.$$
Indeed, let $z_1$ and $z_2$ be two distinct points in $D_{r_0}$ and let $[z_1,z_2]$ denote the line segment from $z_1$ to $z_2$.\\

By (\ref{uzzz}), (\ref{ldu}) we have

\begin{eqnarray*}
	|u(z_1)-u(z_2)|&=&\bigg|\int_{[z_1,z_2]}u_z(z)\,dz+u_{\overline{z}}(z)\,d\overline{z}\bigg|\\
	&\geq & \bigg|\int_{[z_1,z_2]}u_z(0)\,dz+u_{\overline{z}}(0)\,d\overline{z}\bigg|\\
	&-&\bigg|\int_{[z_1,z_2]}(u_z(z)-u_z(0))\,dz+(u_{\overline{z}}(z)-u_{\overline{z}}(0))\,d\overline{z}\bigg|\\
	&\geq &\lambda( D_u (0))|z_1- z_2|\\
	&-&\int_{[z_1,z_2]}(|u_z(z)-u_z(0)|+|u_{\overline{z}}(z)-u_{\overline{z}}(0)|)\,|dz|\\
	&>& |z_2-z_1|\bigg\{ \frac{\pi}{4M}-\frac{4}{\pi}\frac{M r_0}{(1-r_0)^3}\big(-r_0^3+3r_0^2-3r_0+3 \big)\bigg\}\\
	&=&0.
\end{eqnarray*}

This implies that $u$ is univalent on $D_{r_0}$.\\

let $\xi=r_0e^{i\theta}\in \partial\D_{r_0}$. Then we infer from (\ref{uzzz}) that  
\begin{eqnarray*}
	|u(\xi)-u(0)|&\geq  &\lambda( D_\Phi (0))r_0
	-\int_{[0,\xi]}(|u_z(z)-u_z(0)|+|u_{\overline{z}}(z)-u_{\overline{z}}(0)|)|dz|\\
	&\geq& \frac{\pi r_0}{4M}-\frac{4M}{\pi}\frac{r_0^2}{(1-r_0)^3}\int_0^1(  -r_0^3t^4+3r_0^2t^3-3r_0t^2+3t )dt\\
	&=& \frac{\pi r_0}{4M}-\frac{4M}{\pi}\frac{r_0^2}{(1-r_0)^3}\bigg(-\frac{r_0^3}{5}+\frac{3}{4}r_0^2-r_0+\frac{3}{2} \bigg)\\
	&=& \frac{4  Mr_0^2}{\pi(1-r_0)^3}\bigg(-r_0^3+3r_0^2-3r_0+3 -(-\frac{r_0^3}{5}+\frac{3}{4}r_0^2-r_0+\frac{3}{2})\bigg)\\
	&=& \frac{4 M r_0^2}{\pi(1-r_0)^3}\big(-\frac{4}{5}r_0^3+\frac{9}{4}r_0^2-2r_0+\frac{3}{2}\big).\\
\end{eqnarray*}
Hence $u(D_{r_0})$ contains an univalent disk $D_{R_0}$ with
$$R_0 \geq  \frac{4M r_0^2}{\pi(1-r_0)^3}\big(-\frac{4}{5}r_0^3+\frac{9}{4}r_0^2-2r_0+\frac{3}{2}\big). $$
\end{proof}

 \section{Schwarz-Type Lemmas for Solutions to Inhomogeneous Biharmonic Equations}
\begin{proof}[Proof of Theorem \ref{theo3}]
 
The solution of (\ref{ibdp}) can be written in the following form
\begin{equation*}
\Phi(z)=\frac{1}{2}(1-|z|^2)P[f+h](z)+K_2[f](z)-G[g](z).\label{phi1}
\end{equation*}
As $z\longmapsto K_2[f](z) $ is $T_2$-harmonic function, then by Theorem \ref{theo1}, we have 
\begin{equation}\label{31}
\bigg|K_2[f](z)-\frac{(1-|z|^2)^3}{(1+|z|^2)^2}K_2[f](0)\bigg| \leq  \frac{2}{\pi}\bigg[(1+|z|^2)\arctan |z|+ \frac{|z|(1-|z|^2)}{1+|z|^2}\bigg]\| f\|_\infty .
\end{equation}
On the other hand, using  the  estimate (\ref{harm}) for the harmonic mapping $P[f+h]$, we get 
\begin{equation}\label{32}
\bigg|P[f+h](z)-\frac{1-|z|^2}{1+|z|^2}P[f+h](0)\bigg|\leq \frac{4}{\pi}\arctan|z|\,\| f+h\|_\infty.
\end{equation}

\noindent Using \cite[inequality 2.3]{Chel}, we obtain
\begin{equation}\label{33}
|G[g](z)| \leq \frac{(1-|z|^2)^2}{64}\| g\|_\infty .
\end{equation}

\noindent Finally as $\displaystyle  K_2[f](0)=\frac{1}{2}P[f](0)$, then if follows from (\ref{31}--\ref{33}) that
\begin{eqnarray*}
\bigg| \Phi(z)-\frac{1}{2}\frac{(1-|z|^2)^3}{(1+|z|^2)^2}P[f](0)-\frac{1}{2}\frac{(1-|z|^2)^2}{1+|z|^2}P[f+h](0)\bigg|&\leq &\frac{2}{\pi}(1-|z|^2)\arctan|z|\| f+h\|_\infty\nonumber\\
&+&\frac{2}{\pi}\Bigg[(|z|^2+1)\arctan |z|+\frac{|z|(1-|z|^2)}{1+|z|^2}\Bigg]\| f\|_\infty\nonumber\\
&+&\frac{(1-|z|^2)^2}{64}\|g\|_\infty.
\end{eqnarray*}
Hence, the proof is complete. 
\end{proof}

\begin{proof}[Proof of Theorem \ref{theo4}]
	The solution of (\ref{ibdp}) can be written in the following form
	 \begin{eqnarray*}
 \Phi(z)&=&\frac{1}{2} (1-|z|^2)P[f+h](z)+ K_2[f](z) - G[g](z).
\end{eqnarray*}
Thus 
\begin{eqnarray}
 \Phi_z(z)&=&\frac{1}{2} \bigg[(1-|z|^2)[P[f+h](z)]_z-\overline{z} P[f+h](z)\bigg]+ K_2[f]_z(z)-G[g]_z(z), \nonumber
\end{eqnarray}
\begin{eqnarray}
 \Phi_{\overline{z}}(z)&=&\frac{1}{2} \bigg[(1-|z|^2)[P[f+h](z)]_{\overline{z}}-z P[f+h](z)\bigg]+K_2[f]_{\overline{z}}(z)-G[g]_{\overline{z}}(z).\nonumber
 \end{eqnarray}
 Therefore
$$ \| D_\Phi (z)\| \leq  \frac{1}{2}(1-|z|^2)\| D_{P[f+h]} (z)\| +|z||P[f+h](z)| +\| D_{K_2[f]}(z)\|+\| D_{G[g]}(z)\|.$$

By Colonna \cite{Col}, we have

\be\label{colo}
\| D_{P[f+h]}(z) \| \leq \frac{4}{\pi} \frac{1}{1-|z|^2}\|f+h\|_\infty.
\ee
It follows from  \cite[Lemma 2.5]{Li},  that
\begin{equation}\label{25}
\| D_{G[g]}(z)\| \leq \frac{23}{48}\|g\|_\infty,
\end{equation}
since  
$$ \ds \int_\D |G_z(z,\omega) g(\omega)|dA(\omega) \leq \frac{23}{6} \|g\|_\infty \text{ and } \int_\D |G_{\overline{z}}(z,\omega) g(\omega)|dA(\omega) \leq \frac{23}{6} \|g\|_\infty.$$
Therefore, combining  (\ref{colo})-(\ref{25}), we obtain
\begin{eqnarray*}
\| D_\Phi (z)\| &\leq & \frac{1}{2}(1-|z|^2)\| D_{P[f+h]} (z)\| +|z||P[f+h](z)| +\| D_{K_2[f]}(z)\|+\| D_{G[g]}(z)\|\\
&\leq &\frac{2}{\pi}\| P[f+h]\|_\infty +|z| \| f+h\|_\infty + \frac{(2+5|z|)(1+|z|^2)}{1-|z|^2} \| f\|_\infty+\frac{23}{48}\| g\|_\infty \\
&\leq &(\frac{2}{\pi} +|z|)\| f+h\|_\infty +\frac{(2+5|z|)(1+|z|^2)}{1-|z|^2} \|f\|_\infty+\frac{23}{48}\| g\|_\infty.
\end{eqnarray*}
\end{proof}

\begin{proof}[Proof of  Theorem \ref{theo5}]
Suppose that $|z|=r$,
it follows from Theorem \ref{theo3} that
\begin{eqnarray*} 
|\Phi(\eta)-\Phi(r\eta)|&\geq &1-\frac{1}{2}(1-r^2) \| f+h\|_\infty -\frac{2}{\pi}\bigg[(r^2+1)\arctan r+\frac{r(1-r^2)}{1+r^2}\bigg]\\
&-&\frac{\| g\|_\infty (1-r^2)^2}{64}-\frac{1}{2}\frac{(1-r^2)^3}{(1+r^2)^2}|P[f](0)| -\frac{1}{2}\frac{(1-r^2)^2}{1+r^2}|P[f+h](0)|
\end{eqnarray*}
Divide by $1-r$ and used the Hospital rule, we obtain 
\begin{eqnarray*} 
\liminf_{r\to 1}\frac{|\Phi(\eta)-\Phi(r\eta)|}{1-r}&\geq &\lim_{r\to 1}\frac{1-\frac{2}{\pi}(r^2+1)\arctan r}{1-r}-\lim_{r\to 1}\frac{2}{\pi}\frac{r(1-r^2)}{(1-r)(1+r^2)}\\
&-&\frac{1}{2}\lim_{r\to 1}(1+r)\| f+h\|_\infty \\
&=&[\varphi'(r)]_{r=1}-\frac{2}{\pi}-\| f+h\|_\infty,
\end{eqnarray*}
where $$\varphi(r)=\frac{2}{\pi}(r^2+1)\arctan r.$$
Then $$\varphi'(r)=\frac{2}{\pi}\big[2r\arctan r+1\big], \text{ so  } \varphi'(1)=1+\frac{2}{\pi}.$$
Hence 
$$\liminf_{r\longrightarrow 1}\frac{|\Phi(\eta)-\Phi(r\eta)|}{1-r}\geq 1-\| f+h\|_\infty.$$
\end{proof}

\section{A Landau-Type Theorem for Solutions to Inhomogeneous Biharmonic Equations }

 First, let us recall the following result.\\
\\
\begin{Thm}[\cite{Chep}, Lemma 1]
Suppose $f$ is a harmonic mapping of $\D$ into $\C$ such that $|f(z)|\leq M$ for all $z\in \D$ and 
$$f(z)=\sum_{n=0}^\infty a_n z^n+\sum_{n=1}^\infty \overline{b}_n \overline{z}^n.$$ 
Then $|a_0|\leq M$ and for all $n\geq 1,$ 
$$|a_n|+|b_n|\leq \frac{4M}{\pi}.$$
\end{Thm}

\begin{proof}[Proof of Theorem \ref{theo6}]
	The solution of (\ref{ibdp}) can be written in the following form
$$\Phi(z)=H_0[f+h](z)+K_2[f](z)-G[g](z). $$

where
\be\label{h0}
H_0[f+h](z)= \frac{1}{2} (1-|z|^2)P[f+h](z).
\ee
Since $P[f+h](z)$ is harmonic in $\D$,  we have
$$H_0[f+h](z)=\sum_{n=0}^\infty a_n z^n+\sum_{n=1}^\infty \overline{b}_n \overline{z}^n$$
Since $|P[f+h](z)|\leq M_2$ for all $z\in \D$, by Theorem F, we have
\be \label{est}
|a_n|+|b_n|\leq \frac{4M_2}{\pi}  \quad \text{ for } n\geq 1.
\ee
By (\ref{h0}) and (\ref{est}),  we have
$$[H_0[f+h](z)]_z=\frac{1}{2}(1-|z|^2)P[f+h]_z(z)-\frac{1}{2}\overline{z}P[f+h](z),$$
and 
$$[H_0[f+h](z)]_{\overline{z}}=\frac{1}{2}(1-|z|^2)P[f+h]_{\overline{z}}(z)-\frac{1}{2}z P[f+h](z).$$
Thus
\begin{eqnarray}
& & \bigg|H_0[f+h]_z(z)-H_0[f+h]_z(0)\bigg|+\bigg|H_0[f+h]_{\overline{z}}(z)-H_0[f+h]_{\overline{z}}(0)\bigg|\nonumber  \\
&\leq&|z| |P[f+h](z)|+\frac{1}{2} \bigg( \big| P[f+h]_z(z)-P[f+h]_z(0)\big|+P[f+h]_{\overline{z}}(z)-P[f+h]_{\overline{z}}(0)\big| \bigg)\nonumber\\
&+&\frac{1}{2}|z|^2\bigg( \big| P[f+h]_z(z)\big|+\big|P[f+h]_{\overline{z}}(z)\big| \bigg)\nonumber\\
&\leq& M_2|z|+\frac{1}{2}\sum_{n\geq 2}n(|a_n|+|b_n|) |z|^{n-1} +\frac{1}{2}|z|^2 \sum_{n\geq 1}n(|a_n|+|b_n|) |z|^{n-1}  \nonumber\\
&\leq& M_2|z| +\frac{1}{2}(1+|z|^2) \sum_{n\geq 2}n(|a_n|+|b_n|) |z|^{n-1}+\frac{1}{2}|z|^2(|a_1|+|b_1|) \nonumber\\
&\leq & M_2|z|+\frac{2M_2|z|}{\pi}\bigg[\frac{(2-|z|)(1+|z|^2)}{(1-|z|)^2}+|z| \bigg].\label{h0z}
\end{eqnarray}

Since $K_2$ is $T_2$-harmonic, then 
	$$ K_2(z)=\sum_{k=0}^\infty c_k F(-1,k-1;k+1;|z|^2)z^k+
\sum_{k=1}^\infty c_{-k} F(-1,k-1;k+1;|z|^2)\overline{z}^k.$$
Let us denote $$K_2^0[f](z):=K_2(z)-c_0F(-1,-1;1;|z|^2)=K_2[f](z)-c_0(1+|z|^2). $$
Hence
$$ K_2[f]=K_2^0[f](z)+c_0(1+|z|^2).$$

\begin{eqnarray}
& & \bigg|K_2[f]_z(z)-K_2[f]_z(0)\bigg|+\bigg|K_2[f]_{\overline{z}}(z)-K_2[f]_{\overline{z}}(0)\bigg|\nonumber  \\
&\leq&\bigg|K_2^0[f]_z(z)-K_2^0[f]_z(0)\bigg|+\bigg|K_2^0[f]_{\overline{z}}(z)-K_2^0[f]_{\overline{z}}(0)\bigg|+2|c_0||z| \nonumber
\end{eqnarray}
By (\ref{c0}), we have
\be
2|c_0| \leq M_1. \nonumber
\ee
On the other hand, as $K_2^0(f)$ is a $T_2$-harmonic function with $K_2^0(0)=0$, it yields
\be |K_2^0[f]_z(z)-K_2^0[f]_z(0) |+| K_2^0[f]_{\overline{z}}(z)-K_2^0[f]_{\overline{z}}(0)|
\leq \frac{4M_1|z|}{\pi(1-|z|)^3}\big(-|z|^3+3|z|^2-3|z|+3\big).\nonumber
\ee

Thus
\be\label{kef}
\bigg|K_2[f]_z(z)-K_2[f]_z(0)\bigg|+\bigg|K_2[f]_{\overline{z}}(z)-K_2[f]_{\overline{z}}(0)\bigg|\leq  \frac{4M_1|z|}{\pi(1-|z|)^3}\big(-|z|^3+3|z|^2-3|z|+3\big)+M_1|z|.
\ee
Let 
$$\psi_1(z)=\bigg|\frac{1}{16\pi}\int_\D g(\omega)(G_z(z,\omega)-G_z(0,\omega))dA(\omega)\bigg|$$
and 
$$\psi_2(z)=\bigg|\frac{1}{16\pi}\int_\D g(\omega)(G_{\overline{z}}(z,\omega)-G_{\overline{z}}(0,\omega))dA(\omega)\bigg|.$$

Then by \cite[Inequality $(3.6)$]{Chel}, we have 
\begin{equation}
\psi_1(z)\leq \bigg(\frac{1-|z|^2}{16}+\frac{43}{120}\bigg)\| g\|_\infty |z|,\label{psi1}
\end{equation}
and 
$$\psi_2(z)\leq \bigg(\frac{1-|z|^2}{16}+\frac{43}{120}\bigg)\| g\|_\infty |z|.$$

Now, it follows from  (\ref{h0z}), (\ref{kef}) and (\ref{psi1})  that 

\begin{eqnarray}
& &|\Phi_z (z)-\Phi_z (0)|+|\Phi_{\overline{z}} (z)-\Phi_{\overline{z}} (0)|\nonumber\\
&\leq & M_2|z|+\frac{2M_2|z|}{\pi}\bigg[\frac{(2-|z|)(1+|z|^2)}{(1-|z|)^2}+|z| \bigg]+\nonumber\\
&+& \frac{4M_1|z|}{\pi(1-|z|)^3}\big(-|z|^3+3|z|^2-3|z|+3\big)+M_1|z|+\psi_1(z)+\psi_2(z)\nonumber\\
&\leq & \sigma(z),\label{sigg}
\end{eqnarray}
where 
$$\sigma(|z|):=(M_1+M_2+\frac{101}{120} M_3)|z|+\frac{2M_2|z|}{\pi}\bigg[\frac{(2-|z|)(1+|z|^2)}{(1-|z|)^2}+|z| \bigg]+\frac{4M_1|z|}{\pi(1-|z|)^3}\big(-|z|^3+3|z|^2-3|z|+3\big)$$
Remark that not only $\sigma(|z|)$ is increasing but also  $\ds\frac{\sigma(|z|)}{|z|}$ is  increasing with respect to $|z|$ in $[0,1)$.\\

By Theorem \ref{theo4}, we obtain that 
$$1=J_\Phi (0)=\| D_\Phi (0)\|\lambda( D_\Phi (0))\leq \lambda( D_\Phi (0))\bigg(\frac{4}{\pi}M_1+\frac{2}{\pi}M_2+\frac{23}{48}M_3\bigg)$$
yields
\begin{equation}\displaystyle
\lambda( D_\Phi (0))\geq \frac{1}{\frac{4}{\pi}M_1+\frac{2}{\pi}M_2+\frac{23}{48}M_3}.\label{lambd}
\end{equation}

We will prove $\Phi$ is univalent in $\D_{r_0}$, where $r_0$ satisfies the following
equation:
\begin{equation}
(\frac{4}{\pi} M_1+\frac{2}{\pi}M_2+\frac{23}{48}M_3)\sigma(r_0)=1.\label{eq0}
\end{equation}
We choose two points $z_1\neq z_2 \in \D_{r_0}$, 
and by  (\ref{sigg})-(\ref{eq0}), we obtain
\begin{eqnarray*}
|\Phi(z_1)-\Phi(z_2)|&=&\bigg|\int_{[z_1,z_2]}\Phi_z(z)dz+\Phi_{\overline{z}}(z)d\overline{z}\bigg|\\
&\geq & \bigg|\int_{[z_1,z_2]}\Phi_z(0)dz+\Phi_{\overline{z}}(0)d\overline{z}\bigg|\\
&-&\bigg|\int_{[z_1,z_2]}(\Phi_z(z)-\Phi_z(0))dz+(\Phi_{\overline{z}}(z)-\Phi_{\overline{z}}(0))d\overline{z}\bigg|\\
&\geq &\lambda( D_\Phi (0))|z_1- z_2|\\
&-&\int_{[z_1,z_2]}(|\Phi_z(z)-\Phi_z(0)|+|\Phi_{\overline{z}}(z)-\Phi_{\overline{z}}(0)|)|dz|\\
&>& |z_1- z_2|\bigg(\frac{1}{\frac{4}{\pi} M_1+\frac{2}{\pi}M_2+\frac{23}{48}M_3}-\sigma(r_0)\bigg)\\ 
&=&0.
\end{eqnarray*}
Thus, from the arbitrariness of $z_1$ and $z_2$, the univalence of $\Phi$ follows.\\

Now, we will prove $\Phi(\D_{r_0})$ contains an univalent disk $\D_{R_0}$
To reach this goal, let $\xi=r_0e^{i\theta}\in \partial\D_{r_0}$. As the mapping $\frac{\sigma(|z|)}{|z|}$ is increasing, we deduce 

\be
\int_{[0,\xi]} \sigma(|z|)|dz| \leq \frac{\sigma(r_0)r_0}{2}.\nonumber
\ee
Therefore,

\begin{eqnarray*}
|\Phi(\xi)-\Phi(0)|&=&\bigg|\int_{[0,\xi]}\Phi_z(z)dz+\Phi_{\overline{z}}(z)d\overline{z}\bigg|\\
&\geq & \bigg|\int_{[0,\xi]}\Phi_z(0)dz+\Phi_{\overline{z}}(0)d\overline{z}\bigg|\\
&-&\bigg|\int_{[0,\xi]}(\Phi_z(z)-\Phi_z(0))dz+(\Phi_{\overline{z}}(z)-\Phi_{\overline{z}}(0))d\overline{z}\bigg|\\
&\geq &\lambda( D_\Phi (0))r_0 -\int_{[0,\xi]}\sigma(z)|dz|\\
&\geq&  \sigma(r_0)r_0-\frac{\sigma(r_0)r_0}{2}\\
&=&\frac{\sigma(r_0)r_0}{2}\\
&=&\frac{r_0}{\frac{8}{\pi} M_1+\frac{4}{\pi}M_2+\frac{23}{24}M_3}.
\end{eqnarray*}

Hence  $\Phi(\D_{r_0})$ contains an univalent disk $\D_{R_0}$ with the
radius $R_0$ satisfying
$$R_0\geq \frac{r_0}{\frac{8}{\pi} M_1+\frac{4}{\pi}M_2+\frac{23}{24}M_3}. $$
\end{proof}


\begin{thebibliography}{99}

\bibitem{Abd} Abdulhadi, Z., Abu, Y.: Muhanna, Landau's theorem for biharmonic mappings. J. Math. Anal. Appl. {\bf 338} (2008), 705–709.

\bibitem{borh} A. Borichev and H. Hedenmalm, Weighted integrability of polyharmonic functions, Adv. Math., {\bf 264} (2014), 464-505.
\bibitem{Bur} Burns, D. M. and Krantz, S. G.: Rigidity of holomorphic mappings and a new Schwarz lemma at the boundary, J. Amer. Math. Soc., {\bf 7} (1994), 661–676.
\bibitem{Boc} Bochner, S.: Bloch's theorem for real variables, Bull. Amer. Math. Soc., {\bf 52} (1946), 715–719.
\bibitem{Bonk} Bonk, M. and Eremenko, A.: Covering properties of meromorphic functions, negative curvature and spherical geometry, Ann. Math., {\bf 152} (2000), 551–592.

\bibitem{Che1} Chen, H., Gauthier, P.M. and Hengartner, W.: Bloch constants for planar harmonic mappings, Proc. Amer. Math. Soc., {\bf 128} (2000), 3231–3240.

\bibitem{Che2} Chen, H. and Gauthier, P. M.: Bloch constants in several variables, Trans. Amer. Math. Soc., {\bf 353} (2001), 1371–1386.

\bibitem{Chep} Chen, Sh., Ponnusamy, S. and  Wang, X.: Bloch constant and Landau's theorems for planar p-harmonic mappings, J. Math. Anal. Appl., {\bf 373} (2011), 102–110.

\bibitem{Chev}Chen, Sh. and Vuorinen, M.: Some properties of a class of elliptic partial differential operators, J. Math. Anal. Appl., {\bf 431} (2015), 1124–1137.

\bibitem{Chel} Chen, Sh.,Li, Pe., Wang, Xi.:Schwarz-Type Lemma, Landau-Type Theorem, and Lipschitz-Type Space of Solutions to Inhomogeneous Biharmonic Equations, The Journal of Geometric Analysis, doi.org/10.1007/s12220-018-0083-6

\bibitem{ChenZu} S. Chen, J.-F. Zhu, Schwarz type Lemmas and a Landau type theorem of functions satisfying
the biharmonic equation, Bull. Sci. math. (2019), https://doi.org/10.1016/j.bulsci.2019.01.015

\bibitem{Col} Colonna, F.: The Bloch constant of bounded harmonic mappings. Indiana Univ. Math. J. {\bf 38}, 829–840 (1989)

\bibitem{Dub} Dubinin,V. N.: Conformal mappings and inequalities for algebraic polynomials, Algebra Analiz, 13, No. {\bf 5}, 16–43 (2001).

\bibitem{Dub1}  Dubinin,V. N.: On application of conformal mappings to inequalities for rational functions, Izv. Ross. Akad.Nauk, Ser. Mat., {\bf 66}, No. 2 (2002).



\bibitem{Gol} Goluzin, G. M.: Geometric Theory of Functions of Complex Variable [in Russian], 2nd edn., Moscow (1966).

\bibitem{Hei} Heinz, E.: On one-to-one harmonic mappings. Pacific J. Math. {\bf 9} (1959), 101–105.

\bibitem{heth}  H.  W. Hethcote, {\it  Schwarz lemma analogues for harmonic functions},
Int. J.  Math.  Educ.   Sci.   Technol.,  Vol. { \bf 8}, No. 1(1977), 65-67



\bibitem{Kal1} Kalaj, D.: A Sharp Inequality for Harmonic Diffeomorphisms of the Unit Disk. arXiv: $1706.01990$ [math.CV]

 
\bibitem{Kra} Krantz, S. G.: The Schwarz lemma at the boundary, Complex Var. Elliptic Equ., 56 (2011), 455–468.

\bibitem{Lan} Landau, E.: \"Uber die Bloch'sche konstante und zwei verwandte weltkonstanten, Math. Z., {\bf 30} (1929), 608–634

\bibitem{Lav} Lavrent'ev, M. A. and  Shabat, B. V.: Methods of Theory of Functions of Complex Variable [in Russian], 4th edn., Moscow (1973).

\bibitem{Li}Li, P., Ponnusamy, S.: Representation formula and bi-Lipschitz continuity of solutions to inhomogeneous biharmonic Dirichlet problems in the unit disk, J. Math. Anal. Appl. {\bf 456} (2017), 1150-1175

\bibitem{Liu} Liu,T. and  Tang, X.: Schwarz lemma at the boundary of strongly pseudoconvex domain in $\C^n$, Math. Ann., {\bf 366} (2016), 655–666.

\bibitem{Mate} Mateljevi\'c, M., Khalfallah, A.: Schwarz lemmas for mappings with bounded Laplacian, arXiv:$1810.08823$[math.CV]

\bibitem{Olf} Olofsson, A.: Differential operators for a scale of Poisson type kernels in the
unit disc, J. Anal. Math., {\bf 123} (2014), 227-249.

\bibitem{Olf1} Olofsson, A and Wittsten, J.: Poisson integral for standard weighted Laplacians
in the unit disc, J. Math. Soc. Japan, {\bf 65}(2011), 447-486.

\bibitem{Pav} Pavloviç, M.: Introduction to function spaces on the disk. Matemati\'cki institut SANU, Belgrade (2004)

\bibitem{Pom} Pommerenke, Ch.: Boundary Behaviour of Conformal Maps, New York (1992).

\end{thebibliography}
\end{document}